\documentclass[reqno, 11pt, a4paper]{amsart}
\usepackage{amsfonts, amsmath, amssymb, amsthm, color}
\usepackage[hmargin=2.75cm, vmargin=2.9cm]{geometry}
\usepackage{nccmath}
\usepackage[colorlinks=true]{hyperref}

\newtheorem{thm}{Theorem}[section]

\newtheorem{lem}[thm]{Lemma}
\newtheorem{prop}[thm]{Proposition}

\newtheorem{rem}{Remark}%[section]

\newcommand{\p}{\partial}
\newcommand{\N}{\mathbb{N}}
\newcommand{\D}{\Delta}

\newcommand{\vp}{\varphi}
\newcommand{\R}{\mathbb{R}}
\newcommand{\ve}{\varepsilon}
\newcommand{\bve}{{\varepsilon}}
\newcommand{\LL}{\mathcal{L}}
\newcommand{\GG}{\mathcal{G}} 
 
\newcommand{\g}{\mathfrak g}
\newcommand{\M}{\mathcal{M}}

\begin{document}

\title[Singular Yamabe metrics by equivariant reduction]{Singular Yamabe metrics by equivariant reduction}

\author{Ali Hyder}
\address[Ali Hyder]{ Department of Mathematics,
Johns Hopkins University,
3400 N. Charles Street,
Baltimore, MD 21218}
\email{ahyder4@jhu.edu}

\author{Angela Pistoia}
\address[Angela Pistoia]{Dipartimento SBAI,  ``Sapienza" Universit\`a di Roma, via Antonio Scarpa 16, 00161 Roma, Italy}
\email{angela.pistoia@uniroma1.it}

\author{Yannick Sire} \address[Yannick Sire]
{ Department of Mathematics,
Johns Hopkins University,
3400 N. Charles Street,
Baltimore, MD 21218}
\email{sire@jhu.edu}
\begin{abstract}
We construct singular solutions to the Yamabe equation using a reduction of the problem in an equivariant setting. This provides a non-trivial geometric example for which the analysis is simpler than in Mazzeo-Pacard program.  Our construction provides also a non-trivial example of a weak solution to the Yamabe problem involving an equation with (smooth) coefficients.  

\end{abstract}

\date{\today}
\thanks{A. Hyder was supported by SNSF Grant No. P400P2-183866.
A. Pistoia was partially supported  by project Vain-Hopes within the program VALERE: VAnviteLli pEr la RicErca.}
\maketitle

\tableofcontents

 \section{Introduction }

We consider the semilinear elliptic equation
\begin{equation}\label{p}
 -\Delta _\g u+h u=u^{p},\ u>0,\ \hbox{on}\ (\M,\g)
 \end{equation}
where   $(\M,\g)$  is a $n-$dimensional compact Riemannian manifold without boundary, $h$ is a $C^1-$real function on $\M$   s.t. $-\Delta_\g+h$  is coercive
and $p>1$.

We are interested in finding solutions which are singular at $k-$dimensional manifolds for some integer $k\ge1 .$
\\  

 In the critical case, i.e.  $p=2^*_n-1:={n+2\over n-2}$ when the equation \eqref{p} coincides with the Yamabe equation (for $h=R_\g$ the scalar curvature of $\M$), solutions singular at isolated points and at $k-$dimensional manifolds are known provided $k<(n-2)/2$ (see \cite{MPDuke,Mazzeo-Smale,Mazzeo-Pacard96,schoenCPAM}). 
 
 In the present work, we provide a non trivial example of a geometric singular solution, in a much simpler setting than the original construction in \cite{Mazzeo-Pacard96}. Our idea is to rely on an equivariant reduction of the problem like the ones described for instance in \cite{Clapp-Pistoia}.

For any integer $0\le k\le n-3 $ let $2^*_{n,k}={2(n-k)\over n-k-2} $  be the \textit{$(k+1)-$st critical exponent}.
We remark that $2^*_{n,k}=2^*_{n-k,0}$   is nothing but the critical exponent for the Sobolev embedding  ${\mathrm H}^ 1_\g(\M)\hookrightarrow {\mathrm L}^{q}_\g(\M),$  when $(\M,\g)$ is a $(n-k)-$dimensional Riemannian manifold.
In particular,  $ 2^*_{n,0}={2 n\over n- 2}$    is the usual Sobolev critical exponent. 

In order to reduce the problem, we will consider the background manifold $\M$ to be given by a warped product. Let $(M,g)$ and $(K,\kappa)$ be two riemannian manifolds of dimensions $N$ and $k ,$ respectively. Let   $\omega\in C^2(M),$ $\omega> 0$ be a given function.  The warped product $\M = M\times _\omega K$ is the product (differentiable) $n-$dimensional ($n=N+k$) manifold
$M\times K$ endowed with the riemannian   metric $\g=g +\omega^2\kappa.$  The function $\omega$ is called the {\em warping function}.
For example, every surface of revolution (not crossing the axis of revolution) is
isometric to a warped product, with $M$ the generating curve, $K=S^1$ and $\omega(x)$ the distance from $x \in M$ to the axis of revolution.
\\
It is not difficult to check that if $u\in C^2(M\times _\omega K)$   then
\begin{equation}\label{equ2}\Delta _\g u=\Delta _{g } u +\frac{m}{\omega}g(\nabla_g u \nabla _g u)+{1\over \omega^2}\Delta _{\kappa} u.\end{equation}
Assume $h$ is invariant with respect to $K,$ i.e. $h(x,y)=h(x)$ for any $(x,y)\in M\times K.$
If we look for solutions to \eqref{p} which are invariant with respect to $K,$ i.e. $u(x,y)=v(x)$ then by \eqref{equ2}
we immediately deduce that $u$ solves \eqref{p} if and only if $v$ solves
\begin{equation}\label{equ3}
-\Delta _{g } v -{m \over \omega}g \ (\nabla _{g } v, \nabla  _{g } v\ )+h v=v^{p} \quad\hbox{in}\ (M,g ). 
\end{equation}
or equivalently
%\begin{equation}\label{equ4}
$$- \mathrm{div}_{g }\ ( {\omega^N}\nabla _{g } v\ )+ {\omega^N}h v ={\omega^N}v^{p},\ v>0\quad \hbox{in}\ (M,g ).$$
%\end{equation}
 It is clear that if $v$ is a solution to problem \eqref{equ3} which is singular at a point $\xi_0\in M$ then $u(x,y)=v(x)$ is a solution to problem  \eqref{p}
which is singular only on the fiber $\{\xi_0\}\times K$, which is a $k -$dimensional submanifold of $M\times _\omega K.$
It is important to notice that the fiber $\{\xi_0\}\times K$ is   totally geodesic in $ M\times _\omega K$ (and in particular a minimal submanifold of $ M\times _\omega K$) if $\xi_0$ is a critical point of the warping function $\omega.$\\

Therefore, we are lead to consider the more general anisotropic   problem
\begin{equation}\label{equ5}
- \mathrm{div}_{g }\ ( a\nabla _{g } u\ )+ ah u =au^{p},\ u>0\quad \hbox{in}\ (M,g )\end{equation}
where $(M,g)$ is a $N-$dimensional compact Riemannian manifold, $p>1,$ $h\in C^1(M)$ and  $a\in C^2(M)$ with $\min_M a>0$. We will assume that the anisotropic operator $- \mathrm{div}_{g }\ ( a\nabla _{g } u\ )+ ah u $ is coercive in $H^1 (M).$ 
Our main result reads as follows.
\begin{thm} \label{thm-0} If $\frac{N}{N-2}<p<\frac{N+2}{N-2}$, then the problem \eqref{equ5} has a solution which is singular at a  point $\xi_0\in M$. \end{thm}

As a consequence of the previous theorem and the above discussion, we deduce
 
 \begin{thm} \label{thm-2} 
 Assume that $(\mathcal M, \g)$ is a warped product $M \times_\omega K.$ If $0< k<{n-2\over2}$ then there exists a    solution invariant with respect to $K$ of 
% \begin{equation}\label{yamabe}
$$ -\Delta _\g u+R_\g u=u^{n+2 \over n-2},\ u>0,\ \hbox{in}\ (\M,\g)$$
 %\end{equation}
 which is singular on $\left \{ \xi_0 \right \}  \times K$, where $\xi_0$  is any point on $M$. Furthermore, if $\xi_0$ is a critical point of $\omega$ then the submanifold $\left \{ \xi_0 \right \}  \times K$ is minimal in $\mathcal M$.
 \end{thm}

The proof of Theorem \eqref{thm-0} follows the same strategy developed in \cite{Mazzeo-Pacard96}. In particular, we will replace the $N-$dimensional manifold $M$ by a bounded smooth domain $\Omega$ in $\mathbb R^N$ and we will focus on the Dirichlet boundary problem
 \begin{align}\label{eq-1} \left\{\begin{array}{ll}-div(a\nabla u) +ah u=au^p&\quad\text{in }\Omega \\ u=0&\quad\text{on }\partial\Omega\\ u>0&\quad\text{in }\Omega.\end{array}\right.\end{align} Here  $h\in C^1(M)$, $a\in C^2(\bar\Omega)$ with $\min_\Omega a>0$ and the anisotropic operator $ -div(a\nabla u) +ah u$ is coercive in $H^1_0(\Omega)$. We will show the following result
 \begin{thm} \label{thm-1} If $\frac{N}{N-2}<p<\frac{N+2}{N-2}$, then the problem \eqref{eq-1} has a solution which is singular at a  point $\xi_0\in\Omega$. \end{thm}
 The modification in the arguments to solve the problem on the manifold instead  of in the domain are minor and are described in the last section of \cite{Mazzeo-Pacard96}.\\
 
 The paper is organized as follows. The proof of Theorem \ref{thm-1} is carried out in Section \ref{proof} and  relies on  the linear theory studied in Section \ref{lin-th}
 together with a contraction mapping argument developed in Section \ref{con-q}. All the necessary technical tools are contained in Section   \ref{pre}
 and in the Appendix \ref{app}.

 \section{Preliminaries}\label{pre}

\subsection{Function spaces}

%Let $\Sigma$ be a smooth $k$ dimensional submanifold of $\Omega \subset \R^n$ (or a union of submanifolds with different dimensions). For $\sigma>0$ small we let $N_\sigma$ to be the geodesic tubular neighborhood of radius $\sigma$ around $\Sigma$. 
For $\sigma>0$ we let $N_\sigma$ to be the ball $B_\sigma(\xi_0)$. 
 For $\alpha\in(0,1)$, $s\in (0,\sigma)$, $k\in \N\cup\{0\}$ and $\nu\in\R$ we define the seminorms \begin{align} \label{seminorm}|w|_{k,\alpha,s}:=\sum_{j=0}^k s^j\sup_{N_s\setminus N_\frac s2}|\nabla ^jw| +s^{k+\alpha}\sup_{x,x'\in N_s\setminus N_\frac s2} \frac{|\nabla^kw(x)-\nabla^kw(x')|}{|x-x'|^\alpha}, \end{align} and  the weighted H\"older norm  ($\sigma>0$ is fixed) \begin{align} \notag \|w\|_{C^{k,\alpha}_\nu}  :=|w|_{C^{k,\alpha}(\bar\Omega\setminus N_\frac\sigma2)} +\sup_{0<s<\sigma}s^{-\nu}|w|_{k,\alpha,s}.\end{align} The  weighted H\"older  space $C^{k,\alpha}_\nu(\Omega\setminus\Sigma)$ is defined by (here $\Sigma=\{\xi_0\}$) \begin{align} \notag C^{k,\alpha}_\nu(\Omega\setminus\Sigma):=\left\{w\in C^{k,\alpha}_{loc}(\bar \Omega\setminus\Sigma): \|w\|_{C^{k,\alpha}_\nu} <\infty \right\}. \end{align}  The subspace of $C^{k,\alpha}_\nu(\Omega\setminus\Sigma)$ with Dirichlet boundary conditions will be denoted by $$C^{k,\alpha}_{\nu,\mathcal {D}}(\Omega\setminus\Sigma):=\{w\in C^{k,\alpha}_\nu(\Omega\setminus\Sigma):w=0\text{ on }\partial\Omega\}.$$
The space $C^{k,\alpha}_{\nu,\nu'}(\R^N\setminus \{0\})$ is defined by  \begin{align}\notag \|w\|_{C^{k,\alpha}_{\nu,\nu'}(\R^N\setminus \{0\})}:=\|w\|_{C^{k,\alpha}_{\nu}(B_2\setminus \{0\})} +\sup _{r\geq 1}(r^{-\nu'} \|w(r\cdot)\|_{C^{k,\alpha}(\bar B_2\setminus B_1)}).\end{align}  

We now list some useful properties of the space $C^{k,\alpha}_{\nu}(\Omega\setminus\Sigma)$, see e.g. \cite{Mazzeo-Pacard96} and the book \cite{Pacard-Riviere}.

\begin{lem} \label{lem-properties}The following properties hold. \begin{itemize} \item[i)] If $w\in C^{k+1,\alpha}_{\gamma}(\Omega\setminus\Sigma)$ then   $\nabla w\in C^{k,\alpha}_{\gamma-1}(\Omega\setminus\Sigma)$. 
\item[ii)] If  $w\in C^{k+1,0}_{\gamma}(\Omega\setminus\Sigma)$ then   $ w\in C^{k,\alpha}_{\gamma}(\Omega\setminus\Sigma)$ for every $\alpha\in[0,1)$. 
\item[iii)]  For every $w_i\in C^{k,\alpha}_{\gamma_i}(\Omega\setminus\Sigma)$,\, i=1,2, we have 
 $$\|w_1 w_2\|_{k, \gamma_1+\gamma_2,\alpha}\leq C \|w_1 \|_{k, \gamma_1,\alpha} \|w_2\|_{k, \gamma_2,\alpha},$$ for some $C>0$ independent of $w_1,w_2$.
\item[iv)] There exists $C>0$ such that for every   $w\in C^{k,\alpha}_{\gamma}(\Omega\setminus\Sigma)$   with $w>0$ in $\bar \Omega\setminus \Sigma$ we have  $$\|w^p\|_{k, \gamma,\alpha}\leq C \|w \|^p_{k, \gamma,\alpha} .$$

 \end{itemize}  \end{lem}
 
  \subsection{The singular solution}
 
 The building block for our theory is the existence of a singular solution with different behaviour at the origin and at infinity. The following theorem provides such a solution. % We refer the reader to the appendix for a proof of this result. 
 
  \begin{thm}[\cite{Mazzeo-Pacard96}] \label{exists1}
  Suppose that $\frac{N}{N-2}<p<\frac{N+2}{N-2}$. Then for every $\beta>0$ there exists a unique radial solution $u$ to \begin{align} \label{singu11} \left\{\begin{array}{ll} -\D u=u^p\quad\text{in }\R^N\setminus\{0\}\\ u>0 \quad\text{in }\R^N\setminus\{0\}\\
  \lim_{|x|\to0} u(x)=\infty,  \end{array}  \right.\end{align}  such that $$\lim_{r\to\infty}r^{N-2}u(r)=\beta,\quad \lim_{r\to 0^+}r^\frac{2}{p-1}u(r)=c_p:=[k(p,N)]^\frac{1}{p-1},$$ where \begin{align*}k(p,N)&= \frac{2}{p-1}\left(N-\frac{2p}{p-1}\right). \end{align*} \end{thm} 

Let $u$ be a singular radial solution to \eqref{singu11}. Then $u_\ve(x):=\ve^{-\frac{2}{p-1}}u(\frac x\ve)$ is also a solution to \eqref{singu11}. Note that $$u_\ve(x)\leq C(\delta,u) \ve^{N-2-\frac{2}{p-1}}\quad\text{for }|x|\geq\delta, $$ which shows that $u_\ve\to0$ locally uniformly in $\R^N\setminus\{0\}$. Due to this scaling and the asymptotic behavior of $u$ at infinity, for a given $\alpha>0$,  we can find a solution $u_1$ such that
 %\begin{align}\label{normalized-u1}
 $$r^2u_1^{p-1}(r)\leq \alpha\quad\text{ on  }(1,\infty).$$
 %\end{align}

 \subsection{The linearized operator around the singular solution}

We consider the linearized operator $$L_1=\D +pu_1^{p-1}$$ where in polar coordinates we denote $$\D=\frac{\partial^2}{\p r^2}+\frac{N-1}{r}\frac{\p}{\p r}+\frac{1}{r^2}\D_\theta.$$  %$$\D^2=$$ 

Following \cite{Mazzeo-Pacard96}, we recall that   $\gamma_j$ is an indicial root of $L_1$ at $0$ if $L_1(|x|^{\gamma_j}\vp_j)=o(|x|^{\gamma-2})$, where $\vp_j$ is the $j$-th eigenfunction of $-\D_\theta$ on $S^{N-1}$, that is $-\D_\theta\vp_j=\lambda_j\vp_j$, $$\lambda_0=0,\quad \lambda_j=N-1,\quad \text{for }j=1,\dots,N,$$ and so on.   
Setting 
\begin{align}\label{Ap}A_p:=p\lim_{r\to0}r^2u_1^{p-1}(r)=pk(p,N).\end{align}  we have that   \begin{align*}\gamma_j^{\pm} =\frac12\left[2-N\pm\sqrt{(N-2)^2+4(\lambda_j  -A_p)}\right] .\end{align*}
For $\frac{N}{N-2}<p<\frac{N+2}{N-2}$ we have that  ($\Re$ denotes the real part)
% \begin{align}\label{indicial18} 
$$2-N<-\frac{2}{p-1}<\Re(\gamma_0^{-})\leq \frac{2-N}{2}\leq \Re(\gamma_0^{+})<0   $$
%\end{align}  
and 
$$ \gamma_j^{-}<-\frac{2}{p-1} \quad\text{for }j\geq1. $$
 
 Since $\lim_{r\to\infty}r^2u_1^{p-1}(r)=0$, the indicial roots of $L_1$ at infinity are the same as for the $\D$ itself. These values are given by  \begin{align*}\tilde \gamma_j^{\pm} =\frac12\left[2-N\pm\sqrt{(N-2)^2+4\lambda_j}\right] .\end{align*}

 We shall choose     $\mu,\nu$ in the region   \begin{align}\label{mu-nu} \frac{-2}{p-1}<\nu<\min\left\{ \frac{-2}{p-1}+1,\,\Re(\gamma_0^{-})\right\}\leq\frac{2-N}{2}\leq\Re(\gamma_0^{+})<\mu <0
 ,  \end{align} so that $\mu+\nu=2-N$.

 \medskip

% For a function  $w=w(r,\theta)$ we decompose it as  $$ w(r,\theta)=\sum_{j=0}^\infty w_j(r)\vp_j(\theta).  $$

We have the following propositions whose proofs can be found in \cite{Mazzeo-Pacard96}. 
 
  \begin{prop} \label{inj-prop-1}  Let $w\in C^{2,\alpha}_{\mu,0}(\R^N\setminus\{0\})$  be a solution to $L_1w=0$. Then $w\equiv0$. \end{prop}

\begin{prop}\label{inj-prop-2}Let $w\in C^{2,\alpha}_{\gamma,\gamma}(\R^N\setminus\{0\})$ be a solution to $$\D w+\frac{A_p}{r^2}w=0\quad\text{in }\R^N\setminus\{0\},$$ where $A_p$ is given by \eqref{Ap}. If $\gamma$ is  not an indicial root of the operator $\D +\frac{A_p}{r^2}$ then $w\equiv 0$.  \end{prop}

\section{A scheme of the proof} \label{proof}

Let $\Sigma=\{ \xi_0\}\subset \Omega$. To construct a solution to \eqref{eq-1} which is singular precisely at the point $\xi_0$, we start by constructing  an approximate solution to \eqref{eq-1} which is singular exactly on $\Sigma$. For $\sigma>0$ small (to be chosen later) let us first fix a  non-negative cut-off function $\chi  \in C_c^\infty(B_\sigma(\xi_0))$  such that $\chi=1$  in   $B_\frac\sigma2(\xi_0)$.     An approximate solution $\bar u_{\bve}$  is defined by $$\bar u_{ \bve}(x)= \chi(x)u_{\ve}(x-\xi_0)= \ve^{-\frac{2}{p-1}}\chi(x )u_1(\frac{x-\xi_0}{\ve}).$$ 
We shall look for positive solutions of the form $u=\bar u_\ve+v$. Then, $v$ has to satisfy \begin{align}\label{eq-v} L_\bve v+f_\bve+Q(v)=0,\end{align}
where the linear operator $L_\bve$ is  
\begin{equation}\label{l}  L_\bve v:=div (a\nabla v)+a[p\bar u_\bve^{p-1}-h]v,\end{equation}
the error term is 
\begin{equation}\label{f}f_\bve:= div(a\nabla \bar u_\ve)-ah\bar u_\ve+a\bar u_\ve^p,\end{equation}
and the non-linear term $Q$ is
\begin{equation}\label{q}Q(v)=a[|\bar u_\ve+v|^p-\bar u_\ve^p-p\bar u_\bve^{p-1} v] . \end{equation}

To prove existence of solution to \eqref{eq-v} we will use a fixed point argument on the space $C^{2,\alpha}_{\nu,\mathcal{D}}(\Omega\setminus\Sigma)$ for a suitable $\nu$.  
  We note that if  $v\in C^{2,\alpha}_{\nu,\mathcal{D}}(\Omega\setminus\Sigma)$ solves \eqref{eq-v} then by maximum principle we have that $\bar u_\bve +v>0$ in $\Omega$. This  is a simple consequence of the fact that we will choose $\nu>-\frac{2}{p-1}$, and therefore,   $\bar u_\bve+v>0$ in a small neighborhood of $\Sigma$, thanks to the asymptotic behavior of $\bar u_\bve$,  around the origin and the coercivity assumption on $h$. 
\\

First of all, we estimate the size of the error term.
\begin{lem}\label{error}  The error  $f_\bve$  satisfies  $$\|f_\bve\|_{C^{0,\alpha}_{\gamma-2}}\leq C_\gamma \max\left\{ \ve^{1-\gamma-\frac{2}{p-1}},\ve^{N-\frac{2p}{p-1}}\right\} \quad\text{for } \ve>0\text{ small},  $$ for every $\gamma <1-\frac{2}{p-1}$. \end{lem} %$$\|f_\bve\|_{C^{0,\alpha}_{\gamma-2}}\leq C_\gamma \max\left\{ \ve^{2-\gamma-\frac{2}{p-1}},\ve^{N-\frac{2p}{p-1}}\right\} \quad\text{for } \ve>0\text{ small},$$ for every $\gamma <2-\frac{2}{p-1}$. {\color{blue}If we do not assume that $\nabla a(\xi_0)=0$, then

\begin{proof} We only estimate the first term in \eqref{seminorm}, the estimate for the second term should having the same order as of the first one.  It follows that 
%\begin{align}\label{est-uve}
$$u_\ve(x-\xi_0)\approx\left\{ \begin{array}{ll}  |x-\xi_0|^{-\frac{2}{p-1}}&\quad\text{for }|x-\xi_0|\leq\ve \\ \rule{0cm}{.7cm} \frac{\ve^{N-\frac{2p}{p-1}}}{|x-\xi_0|^{N-2}}&\quad\text{for }|x-\xi_0|\geq\ve,\end{array}\right.  $$
%\end{align} 
and 
% \begin{align}
$$|\nabla u_\ve(x-\xi_0)|\leq C\left\{ \begin{array}{ll}  |x-\xi_0|^{-\frac{2}{p-1}-1}&\quad\text{for }|x-\xi_0|\leq\ve \\ \rule{0cm}{.7cm} \frac{\ve^{N-\frac{2p}{p-1}}}{|x-\xi_0|^{N-1}}&\quad\text{for }|x-\xi_0|\geq\ve.\end{array}\right. $$
% \end{align} 
We write $$f_\ve=a[\D \bar u_\ve+\bar u_\ve^p]+\nabla \bar u_\ve\cdot\nabla a-ah\bar u_\ve.$$Since $\chi\equiv 1$ in a small neighborhood of $\xi_0$, we have  $$a|\D\bar u_\ve+\bar u_\ve^p|\leq C\ve^{N-\frac{2p}{p-1}}\quad\text{and }u_\ve(x-\xi_0)|\nabla \chi\cdot \nabla a|\leq \ve^{N-\frac{2p}{p-1}}.$$  Moreover, % Using that $\nabla a(\xi_0)=0$ we obtain $|\nabla a(x)|\leq C|x-\xi_0|$, which leads to
 \begin{align*} |x-\xi_0|^{2-\gamma}|\nabla u_\ve(x-\xi_0)||\nabla a(x)|\leq C\ve^{1-\gamma-\frac{2}{p-1}}, \end{align*}  $$|x-\xi_0|^{2-\gamma}\bar u_\ve(x-\xi_0)\leq C\ve^{2-\gamma-\frac{2}{p-1}}.$$ The lemma follows. % as $N-\frac{2p}{p-1}\geq 2-\gamma-\frac{2}{p-1}$.

\end{proof}

 Next, we   use the linear theory of $L_\bve$ developed in the   Section \ref{lin-th} and, applying the inverse of $L_\bve$, that is $G_\bve$, we rewrite the above equation \eqref{eq-v} as $$v+G_\bve f_\bve +G_\bve Q(v)=0.$$ The crucial fact we shall use is that the norm of $G_\bve$ is uniformly bounded if $\ve$ is sufficiently small.\\ 
  
   By Lemma \ref{error}, the error   $f_\bve $ satisfies  the estimate   $$\|f_\bve\|_{0,\alpha,\nu-2}\leq C\ve^ q,\quad q:=\min\left\{ {N-\frac{2p}{p-1}}, 1-\nu-\frac{2}{p-1}\right\}.$$ Then, there exists $C_0>0$ such that $\|G_\bve f_\bve \|_{2,\alpha,\nu}\leq C_0 \ve ^q$. This suggests to work on the ball $$\mathcal {B}_{\ve,M}=\left\{  v\in C^{2,\alpha}_{\nu} :\|v\|_{2,\alpha,\nu}\leq M\ve ^q\right\},$$ for some $M>2C_0$ large.   In Section \ref{con-q} we shall show that the map $v\mapsto G_\bve[f_\bve+Q(v)]$ is a contraction on the ball  $\mathcal {B}_{\ve,M}$ when $M$ is large and  $\ve$ is small enough.  That will concludes our proof.

 \medskip

  \section{The linear operator $L_{\bve}$}\label{lin-th}
  \subsection{Injectivity of $L_{\bve}$ on $C^{2,\alpha}_{\mu,\mathcal {D}}(\Omega\setminus\Sigma)$} In this section we study injectivity of the linearized operator  $$L_{\bve}w:=div(a\nabla w)+a[p \bar u_\ve^{p-1}-h] w.$$
  We shall use the following notations:  % $\Omega_{\bve}:=\Omega\setminus \cup _{i=i}^KB_{\ve_i}(\Sigma_i)$. For a function $f$ we shall use the notations 
  $$\Omega_{ \bve}:=\Omega\setminus  B_\ve(\xi_0),\quad f^+:=\max\{f,0\},\quad f^-:=\min\{f,0\}.$$  %Since the parameter $\mu$ lies in $(\frac{2-N}{2},0)$, we have that $C^{1,\alpha}_{\mu,\mathcal {D}}(\Omega\setminus\Sigma)\hookrightarrow H^1_0(\Omega)$. 
  
    \begin{lem} \label{maximum-principle}After a suitable normalization of $u_1$, the operator $L_{ \bve}$ satisfies maximum principle in $\Omega_{\bve}$ for $\ve>0$ small. More precisely, if $w\in H^1(\Omega_\bve)$ satisfies  \begin{align*}  \left\{\begin{array}{ll} L_{\bve}w\geq 0&\quad\text{in }\Omega_{\bve} \\ w\leq0&\quad\text{on }\partial \Omega_{ \bve} ,   \end{array}\right.\end{align*} then $w\leq 0$  in $\Omega_\bve$. \end{lem} 
    
  \begin{proof} The crucial fact we shall use  is that the operator $$v\mapsto -div(a\nabla v)+ahv,$$ is coercive, that is, for some $c_0>0$ we have $$\int_{\Omega}a[|\nabla v|^2+hv^2]dx\geq c_0\int_{\Omega}|\nabla v|^2dx\quad \text{for every }v\in H^1_0(\Omega).$$  Since $w\leq 0$ on the boundary $\partial\Omega_\bve$,  by extending $w^+$ by $0$ on $B_\ve(\xi_0)$   we see that $w^+\in H^1_0(\Omega)$. Multiplying the inequality $L_\ve w\geq 0$ by $w^+$, and  then integrating by parts   we obtain \begin{align*}  \int_{\Omega}a[|\nabla w^+|^2+h(w^+)^2-p\bar u_\bve^{p-1}(w^+)^2]dx=0. \end{align*}  We also have $$\int_{\Omega}\frac{(w^+)^2}{|x-\xi_0|^2}dx\leq \frac{4}{(N-2)^2}\int_{\Omega}|\nabla w^+|^2dx.$$  If we normalize $u_1$ in such a way that $$pa\bar u_\bve ^{p-1}\leq c_0\frac{(N-2)^2}{8}\frac{1}{|x-\xi_0|^2}\quad\text{on }\Omega_\bve,$$ then we have $$\int_{\Omega}|\nabla w^+|^2dx=0.$$ We conclude the lemma. 
    \end{proof}

  \begin{rem} The above proof shows that $L_{\bve}$ satisfies maximum principle in   $B_\sigma(\xi_0)\setminus B_\ve(\xi_0)$  for $\ve>0$ small.    \end{rem}
 
 \begin{lem}\label{7.2}Fix $\ve_0>0$ such that $L_{ \bve}$ satisfies maximum principle on $\Omega_{\bve}$. Let $2-N<\gamma<0$ be fixed. Let  $w_\bve$ be a solution to $L_\bve w_\bve=f_\bve$ on $\Omega_\bve$ for some $f_\bve\in C^{0,\alpha}_{\gamma-2}(\Omega_\bve)$, and $0<\ve\leq\ve_0$. Assume that $w_\bve=0$ on $\partial\Omega$. Then there exists $C>0$ such that \begin{align}\label{wve-est}\|w_\bve\|_{2,\alpha,\gamma}\leq C\left( \|f_\bve\|_{0,\alpha,\gamma-2}+ \ve^{-\gamma}\|w_\bve\|_{C^0(\partial B_{\ve}(\xi_0)} \right) .\end{align}  \end{lem} 
 \begin{proof} % Let $\sigma>0$ be as in Section \ref{section-higher} so that $\bar u_\ve$ is supported in $\cup_{i=1}^KB_{\sigma}(\Sigma_i)$. We fix a smooth positive function $\phi$ on $\Omega\setminus \cup \Sigma_i$ such that $\phi(x)=d(x,\Sigma_i)^\gamma$ in each $B_\sigma(\Sigma_i)$.   For simplicity we assume that $\Sigma_i$ is a point $x_i$. 
For  $\phi(x):=|x-\xi_0|^\gamma$  we have  $$\D \phi(x)=c_{N,\gamma}|x-\xi_0|^{\gamma-2},\quad c_{N,\gamma}:=\gamma(N+\gamma-2)<0.$$   Since $$\nabla a\cdot\nabla\phi-ah\phi=O(|x-\xi_0|^{\gamma-1}),$$ for $\sigma>0$ small we have that $$a\D\phi +\nabla a\cdot\nabla\phi-ah\phi\leq \frac{c_{N,\gamma}}{2}a|x-\xi_0|^{\gamma-2}\quad\text{on }B_\sigma(\xi_0).$$This shows that for a suitable choice of $u_1$, we have for some $\delta>0$ $$L_\bve \phi(x)\leq - \delta |x-\xi_0|^{\gamma-2}\quad\text{on } {\bf \Omega}:=B_\sigma(\xi_0)\setminus B_\ve(\xi_0).$$ Therefore, we can choose $c_{1,\bve}\approx \|f_\bve\|_{0,\alpha,\gamma-2}$ so that $$  L_\bve(w_\bve+c_{1,\bve}\phi)\leq0  \quad\text{on } \bf\Omega.$$ We can also choose \begin{align*}c_{2,\bve}\approx  &  \ve^{-\gamma}\|w_\bve\|_{C^0(\partial B_{\ve}(\xi_0))}  +\|w_\bve\|_{C^0(\partial B_{\sigma}(\xi_0))} =: c_{3,\bve}+c_{4,\bve}, \end{align*} so that $$ w_\bve+(c_{1,\bve}+c_{2,\bve})\phi\geq 0  \quad\text{on } \partial\bf\Omega. $$ Then by Maximum principle we have that (to get the other inequality use $-\phi$) $$|w_\bve|\leq  (c_{1,\bve}+c_{2,\bve})\phi  \quad\text{in }\bf\Omega.$$ Since, $L_\ve w_\bve=f_\bve$ in $\Omega_\bve\setminus\bf\Omega$,  and $w_\ve=0$ on $\partial\Omega$, we get that $$|w_\bve (x) | \lesssim (c_{1,\bve}+c_{2,\bve})\quad \text{for }x\in \Omega_\bve\setminus\bf\Omega.$$ We claim that 
$$c_{4,\bve}\lesssim c_{3,\bve}+\|f_\bve\|_{0,\alpha,\gamma-2}.$$ We assume by contradiction that the above claim is false. Then there exists a family of solutions $w_\ell=w_{\bve_\ell}$ to $L_{\bve_\ell}w_\ell=f_\ell$ with $0<\ve_{\ell}<\ve_0$, $f_\ell \in C^{0,\alpha}_{\gamma-2}(\Omega_{\bve_\ell})$, $w_\ell=0$ on $\partial\Omega$ such that \begin{align}\label{assump-33}c_{4,\bve_\ell}=1\quad\text{and } c_{3,\bve_\ell}+\|f_\ell\|_{0,\alpha,\gamma-2}\to0. \end{align} Then, up to a subsequence, $\Omega_{\bve_\ell}\to \tilde\Omega_{\tilde \ve}$, where   $$\tilde\Omega_{\tilde \ve}=\Omega\setminus \{\xi_0\}\quad\text{if }\tilde\ve=0,\quad \text{and  }\tilde\Omega_{\tilde \ve}=\Omega_{\tilde\ve}=\Omega\setminus B_{\tilde\ve}(\xi_0)\quad\text{if }\tilde\ve>0. $$ %Here  $B_{\ve_i}(\Sigma_i)=\Sigma_i$ if $\ve_i=0$ for some $i$. 
From the estimates on $w_\ell$ we see that $w_\ell\to w$  in  $\tilde \Omega_{\tilde\ve}$. Moreover, $w$ satisfies  \begin{align*}\left\{\begin{array}{ll}L_{\tilde \ve}w=div(a\nabla w)-ahw+p\bar u_{\tilde \ve}^{p-1}w=0&\quad\text{in }\tilde\Omega_{\tilde \ve}\\  w=0&\quad\text{on }\partial\Omega.\end{array}\right.\end{align*}  Here,  for $\tilde\ve=0$  the function  $\bar u_{\tilde \ve}$  is considered to be identically zero. 

If $\tilde \ve>0$ then by Lemma \ref{maximum-principle} we get that $w\equiv 0$.  Next we consider the case $\tilde\ve=0$. We have that   $w(x)=O(|x-\xi_0|^\gamma)$,  and hence the singularity at $\xi_0$ is removable (note that $\gamma>2-N$), that is, $L_{\tilde\ve}w=0$ weakly in $\Omega$.    Thus, we can use  coercivity hypothesis on $h$  to conclude that $w\equiv 0$.  This contradicts the first condition in \eqref{assump-33}.

In this way we have that there exists $C>0$ independent of $\bve$, but  depending only on the right hand side  of \eqref{wve-est}  such that  $$|w_\bve|\leq C\phi \quad\text{in }\Omega_\bve. $$ The desired estimate follows from Lemma \ref{Schauder} and a scaling argument (see e.g. \cite  [Chapter 2.2.1]{Pacard-Riviere}).

 \end{proof}

\begin{lem} \label{inj-omega}There exists $\ve_0>0$ sufficiently small such that if  $\ve<\ve_0$ then $$L_\bve:C^{2,\alpha}_{\mu,\mathcal {D}} (\Omega\setminus\Sigma)\to C^{0,\alpha}_{\mu-2}(\Omega\setminus\Sigma)$$ is injective.   \end{lem}
\begin{proof} We assume by contradiction that   $L_{\bve^\ell}$ is not injective for some $\ve^\ell\to0$. Let $w_\ell \in C^{2,\alpha}_{\mu,\mathcal {D}} (\Omega\setminus\Sigma)$ be a  non-trivial solution to $L_{\bve^\ell}w_\ell=0$. 
 We normalize $w_\ell$ so that \begin{align}\label{normalized-18} \max_{\partial\Omega_{\bve^\ell}} \rho(x)^{-\mu}|w_\ell(x)|=
 (\ve^\ell)^{-\mu}\max_{\partial\Omega_{\bve^\ell}}  |w_\ell(x)|=1,\end{align} where $\rho(x)=|x-\xi_0|$ in a small neighborhood of $\xi_0$, and outside it is  a smooth positive function.  Then by Lemma \ref{7.2} we get  that \begin{align}\label{est-35}\sup_{\Omega_{\bve^\ell}} \left(  \rho(x)^{-\mu}|w_\ell(x)| +\rho(x)^{-\mu+1}|\nabla w_\ell(x)|\right)\leq C.\end{align} 
%  First consider the case when $\Sigma$ is a set of finite points.  Assume that the above maximum is achieved  on $\partial B_{\ve_j^\ell}(x_j)$ for some $j$, and upto a translation, assume that $x_j=0$. 
   We set $$\tilde w_\ell(x)=(\ve^\ell)^{-\mu} w_\ell(\ve^\ell x+\xi_0),\quad |x|<R_\ell:=\frac{\sigma}{2\ve^\ell}.$$ Then $$\D\tilde w_\ell (x)+pu_1^{p-1}\tilde w_\ell(x)=f_\ell(x), $$ where  \begin{align*}f_\ell(x)&:=(\ve^\ell)^{2-\mu} \left( hw_\ell -  a^{-1}\nabla a \cdot \nabla w_\ell \right),\quad y:=\ve^\ell x+\xi_0 \\ &=(\ve^\ell)^{2 }h(y)\tilde w_\ell (x)   -  \ve^\ell a(y)^{-1}\nabla a(y) \cdot \nabla \tilde w_\ell(x). \end{align*} It follows from \eqref{est-35} that $\tilde w_\ell\to  \tilde w_\infty$ and $f_\ell\to 0$ in $C^2_{loc}(\R^N\setminus B_1)$  and  $C^1_{loc}(\R^N\setminus B_1)$ respectively. 
   
   Next we show that $\tilde w_\ell$ is bounded in $C^2_{loc}(B_2\setminus \{0\})$. To this end it suffices to prove that $$S_\ell=\sup_{B_2}\left(|x|^{-\mu}|\tilde w_\ell(x)|+|x|^{-\mu+1}|\nabla \tilde w_\ell (x)|\right)\leq C.$$ We assume by contradiction that the above supremum is not uniformly bounded.  Let $0\neq x_\ell \in B_2$  be such that  $$ S_\ell \approx |x_\ell|^{-\mu}|\tilde w_\ell (x_\ell)|+|x_\ell|^{1-\mu}|\nabla \tilde  w_\ell (x_\ell)| .$$  We claim that $|x_\ell|\to 0$. On the contrary, if $x_\ell \to x_\infty\neq0$, then setting  $\bar w_\ell=\frac{w_\ell}{S_\ell}$ we see that $\bar w_\ell\to \bar w_\infty$, where $$ L_1\bar w_\infty=0\quad \text{in }B_2\setminus\{0\},\quad \bar w_\infty\equiv 0\text{ in }B_2\setminus B_1.$$ Therefore, $\bar w_\infty\equiv 0$ in $B_ 2$, which contradicts to $$|x_\infty|^{-\mu}|\bar w_\infty (x_\infty)|+|x_\infty|^{1-\mu}|\nabla\bar w_\infty(x_\infty)|  \approx 1 .$$  Thus we get that $x_\ell\to0$. 
   
   Now we set $$v_\ell(x)=\frac{r_\ell^{-\mu}\tilde w_\ell(r_\ell x)}{S_\ell},\quad r_\ell:=|x_\ell|. $$ Then, for every $\delta>0$ and $\ell$ large we have  $$L_1 v_\ell=o_\ell(1),\quad |x|^\mu |v_\ell| +|x|^{1+\mu}|\nabla v_\ell|\leq C\quad\text{for  }\delta\leq |x|\leq\frac1\delta.$$ Therefore, up to a subsequence, $v_\ell\to v_\infty$ where $v_\infty$ satisfies $$\D v_\infty +\frac{p\kappa(p,N)}{|x|^2}v_\infty=0, \quad |v_\infty (x)|\leq C|x|^{\mu}\quad\text{in }\R^N\setminus\{0\}.$$ Hence, by Proposition \ref{inj-prop-2} we have $v_\infty\equiv 0 $, a contradiction to $\max_{\partial B_1} (|v_\infty|+|\nabla  v_\infty|)\approx 1 $.  This proves that $S_\ell\leq C$, and consequently we obtain that $\tilde w_\ell\to \tilde w_\infty$ in $C^1_{loc}(\R^N\setminus\{0\})$. Then the limit function $\tilde w_\infty$ would satisfy $$\D \tilde w_\infty+p u_1^{p-1}\tilde w_\infty=0\quad\text{in }\R^N\setminus\{0\},\quad \tilde w_\infty\in C^{2,\alpha}_{\mu,\mu}(\R^N\setminus\{0\}).$$ Then by Proposition \ref{inj-prop-1} we have $\tilde w_\infty\equiv 0 $, a contradiction to \eqref{normalized-18}. 

%Next we consider the case of higher dimensional singularity. Let $x_\ell=(r_\ell,\theta_\ell,y_\ell) $  be a point around $\Sigma_j$ for some $j$  such that $$r_\ell^{-\mu}|w_\ell(x_\ell)|+ r_\ell^{2-\mu}|\D w_\ell(x_\ell)| \approx \sup_{\Omega} \rho(x)^{-\mu}|w_\ell|+ \rho(x)^{2-\mu}|\D w_\ell| =:S_\ell. $$  We can also assume that  $r_\ell\leq \ve_j^\ell$, thanks to \eqref{est-35}. We shall take $r_\ell=\ve_j^\ell$ if they are of the same order so that   either $r_\ell=o(\ve_j^\ell)$ or $r_\ell=\ve_j^\ell$.   We choose Fermi coordinates around $y_\infty$ (limit of $y_\ell$) so that $y_\infty=0$, and the coordinates are defined for $|y|<\tau$ for some $\tau>0$. We set $$\tilde w_\ell(r,\theta,y):=\frac{r_\ell^{-\mu}w_\ell(rr_\ell,\theta, r_\ell(y +\tilde{y}_\ell))}{S_\ell},\quad \tilde{y_\ell}:=\frac{y_\ell}{r_\ell}\,\quad 0<r<\frac{\sigma}{r_\ell},\, |y|<\frac{\tau}{2r_\ell}.$$  

 \end{proof}

  \subsection{Uniform surjectivity  of $L_{\bve}$ on $C^{2,\alpha}_{\mu,\mathcal {D}}(\Omega\setminus\Sigma)$} 
  Instead of using  general theory of edge operators as developed in \cite{mazzeoEdge},  we shall use the notes of Pacard \cite{pacard, pacard2} and Pacard-Rivi\`ere \cite{Pacard-Riviere} for edge operators with point singularity.  
Denoting $\rho(x):=|x-\xi_0|$,  the weighted  space $L^2_{\delta}(\Omega\setminus \Sigma)$ is defined by     (we may also simply write $L^2_\delta$ or $L^2_\delta(\Omega)$)
  $$L^2_{\delta}(\Omega\setminus \Sigma):=\left\{ w\in L^2_{loc}(\Omega\setminus\Sigma): \int_{\Omega} \rho^{-2-2\delta}|w|^2 dx<\infty \right\}.$$ Let $L^2_{-\delta}(\Omega\setminus\Sigma)$ be the dual of $L^2_{\delta}(\Omega\setminus \Sigma)$ with respect to the pairing $$L^2_{\delta}(\Omega\setminus \Sigma)\times L^2_{-\delta}(\Omega\setminus \Sigma)\,\ni\, (w_1,w_2)\longrightarrow \int_{\Omega}w_1w_2 \rho^{-2}dx.$$ We note that  the following embedding is continuous  $$C^{k,\alpha}_{\gamma}(\Omega\setminus\Sigma) \hookrightarrow L^2_{\delta}(\Omega\setminus\Sigma)\quad\text{for }\delta<\gamma+\frac{N-2}{2}.$$  
  
  \begin{lem}  Let $w\in L^2_\delta$ be a solution to $$L_\ve w=0\quad\text{in }\Omega\setminus\{\xi_0\}.$$ Then $w\in C^{2,\alpha}_{\delta-\frac{N-2}{2}}(\tilde\Omega)$ for every  $\tilde\Omega\Subset\Omega$.  \end{lem} 
  \begin{proof} For $x_0\in\Omega$  with $d(x_0,\partial\Omega)\geq |x_0-\xi_0|$ we set $$v(x)=w(x_0+Rx),\quad R=\frac12|x-\xi_0|,\quad |x|\leq 1.$$ Then using the elliptic regularity for $v$, namely, $$\|v\|_{C^0(B_\frac12)}\leq C\|v\|_{L^2(B_1)},$$ one obtains that $|w(x)|\leq C|x|^{\delta-\frac{n-2}{2}}$ for $x$ in a small neighborhood of $\xi_0$. In fact, by elliptic regularity, this estimate also holds on compact sets in $\Omega$. The lemma follows   by a scaling argument  and Schauder regularity.  \end{proof}
  
  The natural domain $D(L_\bve)$ of the operator $L_\bve$ is the set of functions $w\in L^2_\delta$  such that the distributional derivative   $L_\bve w $  is in $L^2_{\delta-2}$. More precisely, $w\in D(L_\ve)$ if there exists $f\in L^2_{\delta-2}$ such that $w$ satisfies $L_\ve w=f$ in    the sense of distributions in $\Omega\setminus\Sigma$.  However, in order to identify the adjoint of $L_\bve$ in a natural way, one has to consider a smaller space including the boundary condition $w=0$ on $\partial \Omega$, which well-defined as a trace according to the next estimate. 
  
  Together with Lemma \ref{L2} and  a rescaling argument (see e.g., \cite[Proposition 1.2.1]{pacard}) one can show that  the following elliptic estimate holds:  for $r_0>0$ with $ B_{2r_0}(\xi_0)\subset\Omega$ \begin{align}\label{elliptic-est}\sum_{\ell=1}^2 \|\nabla^\ell w\|_{L^2_{\delta-\ell}( B_{r_0}(\xi_0))}\leq C(\|f\|_{L^2_{\delta-2}( B_{2r_0}(\xi_0))} +\|w\|_{L^2_\delta( B_{2r_0}(\xi_0)) }). \end{align}  In our next lemma  we bound the  weighted norm $ \|w\|_{L^2_\delta}$ by $L^2$ norm of $w$ and the  weighted norm $ \|f\|_{L^2_{\delta-2}}$  for some values of $\delta$. 
  
  \begin{lem} Assume that  $\delta-\frac{N-2}{2}\not\in\{ \Re{\gamma_j^{\pm}}:j=0,1,\dots\}$.  Then  there exists a compact set $K\subset\bar\Omega\setminus \{\xi_0\}$ and $r>0$ such that \begin{align} \| w\|_{L^{2}_{\delta}(B_r(\xi_0))}\leq C(\ve) (\| f\|_{L^{2}_{\delta-2}( \Omega)}+\|w\|_{L^{2}(K)}). \label{est-26}\end{align}   \end{lem}
  \begin{proof}  Let $R>0$ be such that $B_{4R}(\xi_0)\subset\Omega$. Applying  Lemma \ref{lem-crucial} on the ball $B_R(\xi_0)$ we get that $$\|u\|_{L^2_\delta(B_R(\xi_0))} \leq C_1\left[\|h\|_{L^2_{\delta-2}(B_R(\xi_0))}\|+\|\nabla a\cdot \nabla w\|_{L^2_{\delta-2}(B_R(\xi_0))}+\|u\|_{L^2(K)}\right],$$ for some $C_1>0$ and  for some  compact set $K\subset \bar B_R(\xi_0)\setminus\{\xi_0\}$. Since $\|\nabla a\|\in L^\infty(\Omega)$, it follows that  $$\lim_{r\to 0}\frac{\|\nabla a\cdot \nabla w\|_{L^2_{\delta-2}(B_r(\xi_0))}}{\| \nabla w\|_{L^2_{\delta-1}(B_r(\xi_0))}}=0.$$ Therefore, for $r>0$ small enough, the weighted norm $ \|\nabla a\cdot \nabla w\|_{L^2_{\delta-2}(B_r(\xi_0))}$ can be absorbed one the left hand side, thanks to \eqref{elliptic-est}. On the region $B_R(\xi_0)\setminus B_r(\xi_0)$, the weighted norm $\|\nabla a\cdot \nabla w\|_{L^2_{\delta-2}}$ is equivalent to  $ \|\nabla a\cdot \nabla w\|_{L^2}$, and this can be controlled by $$\|w\|_{L^2(B_{2R}(\xi_0)\setminus B_\frac r2(\xi_0)}+\|f\|_{L^2(B_{2R}(\xi_0)\setminus B_\frac r2(\xi_0)}. $$
  We conclude the lemma. 
   \end{proof}

  As a consequence of \eqref{elliptic-est}-\eqref{est-26} one can prove the following lemma (see e.g.,  Chapter 9, \cite{pacard2}). 
  \begin{lem} The operator $L_\ve:L^2_{\delta}\to L^2_{\delta-2}$ is Fredholm, provided    $\delta-\frac{N-2}{2}\not\in\{ \Re{\gamma_j^{\pm}}:j=0,1,\dots\}$.    \end{lem}

%   In particular, $L_\bve: L^2_\delta\to L^2_{\delta-2}$ is densely defined, and it has closed graph. In particular,   $L_\bve$ is Fredholm (see \cite{mazzeoEdge}). %   For this choice of $\delta$, the following crucial estimate holds: $$\|w\|_{L^2_{-\delta}}\leq C(\|L_\bve w\|_{L^2_{-\delta-4}}+\|w\|_{L^2(\mathcal{K})}),$$ for some  fixed compact set   $\mathcal{K}\subset \bar\Omega\setminus \cup_{i=1}^K\Sigma_i$ independent of $w$. }

    We shall fix $\delta>0 $ slightly bigger than $\mu+\frac{N-2}{2}$, % simply in the region $$\mu+\frac{N-4}{2}<\delta<\frac{N-4}{2},$$
      where $\mu$ is fixed according to \eqref{mu-nu}. Thanks to the previous comment on the domain of $L_\bve$, the adjoint of the operator \begin{align}\label{op}L_\bve :L^2_{-\delta}\to L^2_{-\delta-2}\end{align} is given by \begin{align} \label{op-adj}L^2_{\delta+2}\to L^2_{\delta},\quad  w\mapsto \rho^2L_\bve(w\rho^{-2}).\end{align} Then the adjoint operator    \eqref{op-adj} is injective, and $L_\bve$ in \eqref{op} is surjective.  Using the isomorphism  $$\rho^{2\delta}:L^2_{\tilde \delta}\to L^{2}_{2\delta+\tilde\delta},\quad w\mapsto \rho^{2\delta}w,$$ we identify the adjoint operator as $$L_\bve^*: L^2_{-\delta+2}\to L^2_{-\delta},\quad w\mapsto \rho^{2-2\delta}L_\bve(w\rho^{2\delta-2}) .$$  Now we consider  the composition $$\LL=L_\bve \circ L_\bve^* :L_{-\delta+2}^2\to L_{-\delta-2}^2, \quad w\mapsto L_\bve[\rho^{2-2\delta}L_\bve(w\rho^{2\delta-2})].$$ Then $\LL$ is an isomorphism, and hence there exists a two sided inverse  $$\GG_\bve:L^2_{-\delta-2}\to L^2_{-\delta+2}.$$ Consequently, the right inverse of $L_\bve$ is given by $G_\bve:=L_\bve^*\GG_\bve$.  It follows   that $$G_\bve: C^{0,\alpha}_{\nu-2} (\Omega\setminus\Sigma)\to C^{2,\alpha}_{\nu,\mathcal{D}}(\Omega\setminus\Sigma)$$
is bounded.

\begin{lem} Let $\ve_0>0$ be as in Lemma \ref{inj-omega}. Then for $0<\ve<\ve_0$  the system  $L_\bve w_1=0$, $w_1=L_\bve^* w_2$ with $w_1\in C^{2,\alpha}_{\nu,\mathcal{D}}(\Omega\setminus\Sigma)$ and $w_2\in C^{4,\alpha}_{\nu+2,\mathcal{D}}(\Omega\setminus\Sigma)$ has only   trivial solution.  \end{lem}
\begin{proof} We set $w=\rho^{2\delta-2}w_2$. Then $L_\bve[\rho^{2-2\delta} L_\bve w]=0$. Multiplying the equation by $w$ and then integrating by parts we get $$0=\int_{\Omega}\rho^{2-2\delta}|L_\bve w|^2dx.$$ Since $\nu+2\delta>\mu$, we have  $w\in C^{2,\alpha}_{\nu+2\delta}(\Omega\setminus\Sigma)\subset C^{2,\alpha}_{\mu}(\Omega\setminus \Sigma)$. Then by  Lemma \ref{inj-omega} we get that $w=0$, equivalently  $w_1=w_2=0$.  \end{proof}

\begin{lem}  There exists $\ve_0>0 $ small such that  if $0<\ve<\ve_0$, then the  sequence of solutions $(w_{1,\bve^\ell})\subset C^{2,\alpha}_{\nu,\mathcal{D}}(\Omega\setminus\Sigma) \cap L^*_{\bve^\ell}[C^{4,\alpha}_{\nu+2,\mathcal{D}}(\Omega\setminus\Sigma)]$  to  $L_{\bve^\ell}w_{1,\bve^\ell}=f_{\bve^\ell}$ is uniformly bounded in $C^{2,\alpha}_{\nu}(\Omega\setminus\Sigma)$, provided $(f_{\bve^\ell})$ is uniformly bounded in $C^{0,\alpha}_{\nu-2}(\Omega\setminus\Sigma)$. \end{lem}
\begin{proof} Assume by contradiction that the lemma is false. Then there exists a sequence    $\bve^\ell\to 0$   and  $w_{1,\bve^\ell }\in  C^{2,\alpha}_{\nu,\mathcal{D}}(\Omega\setminus\Sigma)\cap L^*_{\bve^\ell}[C^{4,\alpha}_{\nu+4,\mathcal{D}}(\Omega\setminus\Sigma)]$  with   $L_{\bve^\ell}w_{1,\bve^\ell}=f_{\bve^\ell}$ such that $\|f_{\bve^\ell}\|_{C^{0,\alpha}_{\nu-2,\mathcal{D}}(\Omega\setminus\Sigma)}\leq C$, and $(w_{1,\ve^\ell})$ is not bounded in $C^{2,\alpha}_\nu(\Omega\setminus\Sigma)$.  By Lemma \ref{7.2} $$\|w_{1,\bve^\ell}\|_{C^{2,\alpha}_{\nu}(\Omega_{\bve^\ell})}\leq C+C \max_{ \partial B_{\ve^\ell}(\xi_0)}   (\ve^\ell)^{-\nu}\left( |w_{1,\bve^\ell}|+\ve^\ell|\nabla w_{1,\bve^\ell}|\right) =:C+CS_{\bve^\ell}.$$ 

 We distinguish the following two cases. \\

\noindent\textbf{Case 1} $S_{\bve^\ell}\leq C$.

In this case we proceed as in the proof of Lemma \ref{inj-omega}. Let $x_\ell \in B_{\ve^\ell}(\xi_0)$ be such that $$\sup_{  B_{\ve^\ell}(\xi_0)}  \left(   \rho^{-\nu}|w_{1,\bve^\ell}| + \rho^{-\nu+1}|\nabla w_{1,\bve^\ell}| \right)\approx   |x_\ell-\xi_0|^{-\nu}\left(|w_{1,\bve^\ell}(x_\ell)|+|x_\ell-\xi_0| |\nabla w_{1,\bve^\ell}(x_\ell)|\right) =:S_\ell\to\infty.$$  Then necessarily $r_\ell :=|x_\ell -\xi_0|=o(\ve^\ell)$. Setting $$\tilde w_{1,\bve^\ell}( x):=\frac{r_\ell^{-\nu}w_{1,\bve^\ell}(r_\ell x +\xi_0)}{S_\ell}$$ one would get that $ \tilde w_{1,\bve^\ell}\to \tilde w_1\not\equiv 0$ where $$\tilde L_1\tilde w_1=0\quad\text{in }\R^n\setminus\{0\}, \quad r^{-\nu}|\tilde w_1|\leq C,\quad \tilde L_1:=\D+\frac{A_p}{r^2}, $$ where $A_p$ is as in \eqref{Ap}. Since $\nu$ does not coincide with indicial roots of $\tilde L_1$, from Proposition \ref{inj-prop-2} we  get  that  $\tilde w_1\equiv0$,   a contradiction. \\

\noindent\textbf{Case 2} $S_{\bve^\ell}\to\infty$.

In this case we set $$\tilde w_{1,\ve^\ell}(x)=(\ve^\ell)^{-\nu}\frac{w_{1,\ve^\ell}(\ve^\ell x+\xi_0)}{S_{\ve^\ell}}.$$ Then $\max_{\partial B_1}(|\tilde w_{1,\ve^\ell}|+| \nabla\tilde w_{1,\ve^\ell}|)\approx 1. $
Moreover, %In this case  we first divide the function $\tilde w_{1,\bve^\ell}$ by $S_{\bve^\ell}$. Then consider the scaling with respect to $\ve^\ell$ instead of $r_\ell$  and 
proceeding as before (see Lemma \ref{inj-omega}) we would get that  $w_{1,\bve^\ell}\to \tilde w_1\not\equiv 0$ where $$L_1\tilde w_1=0\quad\text{in }\R^n\setminus\{0\}, \quad r^{-\nu}|\tilde w_1|\leq C. $$ Since $\tilde w_1$ decays at infinity, its decay rate is determined by the indicial roots of $L_1$ (which are exactly the same as $\D$) at infinity. In fact, $\tilde w_1$ would be bounded by $r^{2-N}$ at infinity, see e.g., \cite{Mazzeo-Pacard96}.  % the proof of Proposition \ref{inj-prop-1}.......   

 Since $ w_{1,\bve^\ell}\in L^*_{\bve^\ell}[C^{4,\alpha}_{\nu+2,\mathcal{D}}(\Omega\setminus\Sigma)]$,  we have $w_{1,\bve^\ell}=\rho^{2-2\delta}L_{\bve^\ell} w_{2,\bve^\ell}$ for some $w_{2,\bve^\ell}\in C^{4,\alpha}_{\nu+2\delta,\mathcal{D}}(\Omega\setminus\Sigma)$.  Now we set $$\tilde w_{2,\bve^\ell}(x):=\frac{(\ve^{\ell})^{-\nu-2\delta}w_{2,\bve^\ell}(\ve^{\ell} x +\xi_0)}{S_{\ve^\ell}}. $$ Using that  $2\delta+\nu>\mu$,  and following the  proof of  Lemma  \ref{inj-omega},  one can show that the  family   $\tilde w_{2,\ve^\ell}$ converges to a limit function $\tilde w_2$, where $$ L_1\tilde w_2= |x|^{2\delta-2}\tilde w_1\quad\text{in }\R^n\setminus\{0\}, \quad |x|^{-\nu-2\delta}|\tilde w_2|\leq C.$$ Thus, $ L_1[r^{2-2\delta} L_1 \tilde w_2]=0$. We multiply this equation by $\tilde w_2$ and integrate it on $\R^N$. Then  an integration by parts leads   $ L_1 \tilde w_2=0$ (this is justified because of the decay of $\tilde w_1$ at infinity, provided we choose $\delta>0$ sufficiently close to $\mu+\frac{N-2}{2}$). Again,  as $2\delta+\nu>\mu$,  by Proposition \ref{inj-prop-1} we have   $\tilde w_2=\tilde w_1=0$, a contradiction.

%When $\Sigma$ is of positive dimension,  we need to do the  scaling  as in Lemmas \ref{inj-omega} and \ref{7.3}. We now prove that the limit is independent of the variable $y$. The argument is based on the theory of edge operators and their parametrices\footnote{We are very grateful to Rafe Mazzeo for explaining us the argument, already mentioned in \cite{Mazzeo-Pacard96}}. 

%As in the previous section, we set $$\tilde w_{1,\bve^\ell}(r,\theta,y):=\frac{r_\ell^{-\nu}w_{1,\bve^\ell}(rr_\ell,\theta, r_\ell(y +\tilde{y}_\ell))}{S_\ell},\quad \tilde{y_\ell}:=\frac{y_\ell}{r_\ell}\,\quad 0<r<\frac{\sigma}{r_\ell},\, |y|<\frac{\tau}{2r_\ell}.$$  

%We now proceed as before to show, because of the normalization, that $\tilde w_{1,\bve^\ell}\to \tilde w_1\not\equiv 0$ and  $$\tilde L_1\tilde w_1=0\quad\text{in }\R^n\setminus \R^k, \quad r^{-\nu}|\tilde w_1|\leq C,\quad \tilde L_1:=\D^2-\frac{A_p}{r^4}. $$ 
\end{proof}

  \section{The non-linear term  $Q$}\label{con-q}

 \begin{lem}  Let $M_1>1$ be fixed. Then for $\ve_0<<1$ we have $$\|Q(v_1)-Q(v_2)\|_{0,\alpha,\nu-2}\leq\frac{1}{M_1}\|v_1-v_2\|_{2,\alpha,\nu}$$
 for every $v_1,\,v_2\in \mathcal {B}_{\ve,M}:= \left\{  v\in C^{2,\alpha}_{\nu} :\|v\|_{2,\alpha,\nu}\leq M\ve ^q\right\}.$  \end{lem}
 \begin{proof} In Lemma \ref{error}, the error term $f_\bve$ is bounded  by the maximum of two terms. 
 If the maximum is the second term  $\ve^{N-2p/(p-1)}$, we argue as in   \cite{Mazzeo-Pacard96}. Let us consider the case when the maximum is the first term. Let $$ q_1:=\left(N-\frac{2p}{p-1}\right)-\left( 1-\nu-\frac{2}{p-1}  \right)=N+\nu-3>0.$$We start by showing that there exists $\tau>0$ small  (independent of $\ve<<1$) such that \begin{align}\label{est-tau-1}|v(x)|\leq \frac{1}{10}\bar u_\bve(x)\quad \text{for every }x\in   B_{\tau_\ve}(\xi_0),\, v\in \mathcal {B}_{\ve,M},\end{align} where  $$\tau_\ve:=\tau \ve^\frac{q_1}{\nu-2+N}\to0.$$ To prove this we recall that    there exists $c_1,\,c_2>1$ such that $$\frac{1}{c_1}\leq |x|^\frac{2}{p-1}u_{\ve}(x)\leq c_1\quad \text{for }|x|\leq \ve,$$ $$\frac{1}{c_2}\leq \ve^{-N+\frac{2p}{p-1}}|x|^{N-2}u_{\ve}(x)\leq c_2\quad\text{for }\ve\leq  |x|\leq\tau .$$ On the other hand, $$\ve^{-N+\frac{2p}{p-1}}  \rho(x) ^{-\nu} |v(x)|\leq M .$$ As $\nu>\frac{-2}{p-1}$, we have \eqref{est-tau-1} for some $\tau>0$ small. 
  
 We have  \begin{align*}  Q(v_1)-Q(v_2)&=a\int_0^1\frac{d}{dt}|\bar u_\bve +v_1+t(v_1-v_2) |^pdt -p\bar u_\bve^{p-1}(v_1-v_2)\\ &=ap(v_1-v_2)\int_0^1\left( |\bar u_\bve +v_1+t(v_1-v_2) |^{p-1}-\bar u_\bve^{p-1}\right)dt \\  &=: ap(v_1-v_2)\int_0^1Q(v_1,v_2)dt. \end{align*}Next,  using that $$ (1+r)^{p-1}=1+O(|r|)\quad \text{for }  |r|\leq\frac{1}{2},$$  
  we estimate for $ x\in  B_{\ve}(\xi_0)$  $$\begin{aligned}   |Q(v_1,v_2)|(x) & \leq C\bar u_\bve(x)^{p-1}\frac{|v_1|(x)+|v_2|(x)}{\bar u_\bve(x)}    \notag \\ &\leq CM \ve^{\frac{2}{p-1}+\nu+q} \rho^{-2}(x) \\ &=CM\ve \rho(x)^{-2} , \notag  \end{aligned}$$ and for  $ x\in  B_{\tau_\ve}(\xi_0)\setminus B_{\ve }(\xi_0)$  \begin{align*}  \notag |Q(v_1,v_2) |(x) & \leq C M \rho(x)^{-2} \max\{\ve,  \ve ^{(N-\frac{2p}{p-1})(p-2)+q}\tau_\ve^{(2-N)(p-2)+\nu+2} \}\\ &=CM\rho(x)^{-2}o_\ve(1)  . \end{align*}  Here we have used that the second term in the maximum  is  of the order $\ve^r$ for some $r>0$. Indeed,  from  the  definition of $\tau_\ve$, $q$ and $q_1$, the    exponent of $\ve$ is  \begin{align*} &\left(N-\frac{2p}{p-1}\right)(p-2)+q+[(2-N)(p-2)+\nu+2]\left(1-\frac{1}{N+\nu-2}\right) \\&=1-\frac{(2-N)(p-2)+\nu+2}{N+\nu-2} \\&=\frac{(N-2)(p-1)-2}{N+\nu-2} \\&>0,\end{align*} where the last inequality follows from $p>\frac{N}{N-2}$ and $N+\nu-2>0$. 
 Finally,  as $\nu>-\frac{2}{p-1}$, we easily obtain for $ x\in \Omega\setminus B_{\tau_\ve}(\xi_0)$  $$\begin{aligned}|Q(v_1,v_2)|(x) &\leq C(\bar u_\bve^{p-1}+|v_1|^{p-1}+|v_2|^{p-1} )(x)  \notag \\ &\leq  C \rho(x)^{-2} \left( \bar u_\bve^{p-1}(x)\rho(x)^2+ M\ve^{q(p-1)}\right)\\&=o_\ve(1)\rho(x)^{-2} . \notag \end{aligned} $$  Combining these estimates we  get   for $\ve<<1$ $$\|Q(v_1)-Q(v_2)\|_{0,0,\nu-2} =o_{\ve}(1)\|v_1-v_2\|_{0,0,\nu}=o_{\ve}(1)\|v_1-v_2\|_{2,\alpha,\nu}.$$  
  
% In order  to estimate the weighted  H\"older norm of $Q(v_1)-Q(v_2)$ we note that the function $|\bar u_\bve+v|^{p-1}$ is only $C^{0,p-1}$ for   $1<p<2$, which in turn implies that $Q(v_1,v_2)$ is only $C^{0,p-1}$.  
Next we estimate the weighted H\"older norm of $Q(v_1)-Q(v_2)$   with  H\"older exponent $\alpha\leq p-1$.  
 For $0<s<\sigma$ we write \begin{align*} &s^{2-\nu+\alpha}\sup_{x,x'\in N_s\setminus N_\frac s2}\frac{|[Q(v_1)-Q(v_2)](x)-[Q(v_1)-Q(v_2)](x')|} {|x-x'|^{\alpha}} \\ &\leq 4 \|Q(v_1)-Q(v_2)\|_{0,0,\nu-2} \\&\quad +s^{2-\nu+\alpha}\sup_{x,x'\in N_s\setminus N_\frac s2,\, |x-x'|\leq\frac s4 }\frac{|[Q(v_1)-Q(v_2)](x)-[Q(v_1)-Q(v_2)](x')|} {|x-x'|^{\alpha}} . \end{align*} Notice that for $x,x'\in N_s\setminus N_\frac s2$ with $|x-x'|\leq \frac s4$, the line segment $[x,y]$ joining $x$ and $y$ lies in $N_{2s}\setminus N_\frac s4$. The desired estimate follows on the  ball   $ B_{\tau_\ve}(\xi_0)$  by estimating $Q(v_1,v_2)(x)-Q(v_1,v_2)(x')$ using the following  gradient bound (we are using that $|\bar u_\bve+v|^{p-1}$ is $C^1$ in this region) \begin{align*}\nabla Q(v_1,v_2)&=(p-1)\left[ (\bar u_\bve+v_1+t(v_1-v_2)^{p-2}-\bar u_\bve^{p-2}) \right]\nabla \bar u_\bve  \\ &\quad +  (p-1) (\bar u_\bve+v_1+t(v_1-v_2)^{p-2}\nabla[v_1+t(v_1-v_2)] \\ &=O(1)\bar u_\bve^{p-3}(|v_1|+|v_2|)|\nabla u_\bve|+O(1) \bar u_\bve^{p-2}(|\nabla v_1|+|\nabla v_2|).\end{align*}  In fact,  gradient bounds can  also be used for the region $  B_{\tau_\ve}(\xi_0)$ if $p\geq 2$. For $1<p\leq 2$,   one can use the following inequalty  $$||\phi|^{p-1}(x)-|\phi|^{p-1}(x')|\leq |\phi(x)-\phi(x')|^{p-1}\leq \|\nabla \phi\|^{p-1}_{C^0([x,x'])}|x-x'|^{p-1},$$ with $\phi=\bar u_\bve$ and $\phi=\bar u_\bve+v_1+t(v_1-v_2)$.   

We conclude the lemma. 

 \end{proof}
 
%  \section{Sign changing solutions} Main  equation:  \begin{align} \left\{\begin{array}{ll}-div(a\nabla u) +ah u=a|u|^{p-1}u&\quad\text{in }\Omega \\ u=0&\quad\text{on }\partial\Omega.\end{array}\right.\end{align} Want to construct singular solutions with singularity at $\xi_1$ and $\xi_2$. At $\xi_1$, solution would go to $+\infty$, at $\xi_2$ solution would go to $-\infty$.  For $0<\sigma<\frac14|\xi_1-\xi_2|$ sufficiently  small  let us first fix a  non-negative cut-off function $\chi  \in C_c^\infty(B_\sigma)$  such that $\chi=1$  in   $B_\frac\sigma2(\xi_0)$.     An approximate solution $\bar u_{\bve}$  is defined by $$\bar u_{ \bve}(x)=   \ve^{-\frac{2}{p-1}}\left( \chi(x-\xi_1 )u_1(\frac{x-\xi_1}{\ve}) -\chi(x-\xi_2 )u_1(\frac{x-\xi_2}{\ve}) \right).$$  Error estimates:  as in Lemma \ref{error}.    Linearized operator: $$ L_\bve v:=div (a\nabla v)+ap|\bar u_\bve|^{p-1} v-ahv.$$   Section 4 seems to work without any difficulty. Blow-up analysis around the point $\xi_2$ is exactly same as at $\xi_1$. 
 
   \section{Appendix}\label{app}
  
 % Here we recall some important lemmas from \cite{pacard}. 
  The following lemma can be proven in the spirit of \cite[Proposition 1.5.1]{pacard}
  
  \begin{lem}  \label{lem-7.1}  For   $ d \in\R$ set   
 \begin{align}\label{deltaj}
 \delta_j:=\Re\left( \left(\frac{N-2}{2}\right)^2+\lambda_j-d\right)^\frac12,\quad j\in\N. 
  \end{align} 
  Then for $\delta\in\R\setminus\{\pm \delta_j:j=0,1,\dots\}$ there exists $C=C(N,\delta)$ such that  if $u$ is a solution to $$\D u+\frac{d}{|x|^2} u=f\quad\text{in }B_1\setminus \{0\},$$  then $$\|u\|_{L^2_\delta(B_1)}\leq C(\|f\|_{L^2_{\delta-2}(B_1)}+\|u\|_{L^2(B_1\setminus B_\frac12)}).$$\end{lem}
% \begin{lem} [Exercise 2.2.3]     Let $d:B_1\to\R$ be such that $$d(x)=d_1+d_2(x),\quad |d_2(x)|\leq C|x|^\ve\quad\text{for some }\ve>0.$$   Set $$\delta_j=\Re\left( \left(\frac{N-2}{2}+j\right)^2+d_1\right)^\frac12,\quad j\in\N.$$ Then for $\delta\in\R\setminus\{\delta_j:j=0,1,\dots\}$ there exists $C=C(N,\delta)$ such that  if $u$ is a solution to $$|x|^2\D u+d u=f\quad\text{in }B_1\setminus \{0\},$$  then $$\|u\|_{L^2_\delta(B_1)}\leq C(\|f\|_{L^2_\delta(B_1)}+\|u\|_{L^2(B_1\setminus B_\frac12)}).$$\end{lem}
 
\begin{lem}\label{lem-crucial} Let $\zeta$ be a continuous  function in $\bar B_1$.  Let $\delta_j$ be given by \eqref{deltaj} with $d=\zeta(0)$. Then for $\delta\in\R\setminus\{\pm \delta_j:j=0,1,\dots\}$ there exists  a compact set $K\subset\bar B_1\setminus\{0\}$ and a constant $C>0$ such that  for every $u\in L^2_{\delta}(B_1)$  solving
%\begin{align}\label{apend-32}
$$\D u+\frac{\zeta}{|x|^2}u=f\quad\text{in }B_1\setminus \{0\}, \quad f\in L^2_{\delta}(B_1),  
$$
%\end{align}
  we have   $$\|u\|_{L^2_\delta(B_1)}\leq C(\|f\|_{L^2_{\delta-2}(B_1)} +\|u\|_{L^2( K)}). $$\end{lem} 
\begin{proof} We rewrite the equation as $$\D u+\frac{\zeta(0)}{|x|^2}u=f+\tilde f,\quad \tilde f:=\frac{\zeta(0)-\zeta}{|x|^2}u.$$ Then by Lemma \ref{lem-7.1} we get \begin{align*} \|u\|_{L^2_\delta(B_1)}\leq C_1(\|f\|_{L^2_{\delta-2}(B_1)}+\|\tilde f\|_{L^2_{\delta-2}(B_1)}+\|u\|_{L^2(B_1\setminus B_\frac12)}). \end{align*} Let $r>0$ be sufficiently small so  that $|\zeta-\zeta(0)|\leq\frac{1}{2C_1}$ on $B_r$. Then \begin{align*} \|\tilde f\|_{L^2_{\delta-2}(B_1)}\leq \frac{1}{2C_1}\|u\|_{L^2_{\delta}(B_r)}+C(r,\|\zeta\|_{L^\infty(B_1)})\|u\|_{L^2(B_1\setminus B_r)}. \end{align*} The proof follows by absorbing the term $\|u\|_{L^2_{\delta}(B_r)}$ on the left hand side, and taking $K=\bar B_1\setminus B_r$.

%Using the spherical harmonics we decompose  $$u=\sum_{j=0}^\infty  u_j(r)\vp_j(\theta) ,\quad f=\sum_{j=0}^\infty f_j(r)\vp_j(\theta) .$$ Then $u_j$ satisfies  {\color{blue}(in the presence of a term like  $\frac{\partial u}{\partial x_i}$ in the PDE \eqref{apend-32}, what would be the equation for $u_j$?)}\begin{align} u_j'' +\frac{n-1}{r}u_j'-\frac{\lambda_j-\zeta}{r^2} u_j=f_j .\end{align}  We fix a radial function $\chi \in C_c^\infty(B_1)$ such that $\chi=1$ on $B_{r_0}$ for some $r_0>0$ small (to be chosen later). We multiply the above  equation by $u_j \chi^2 r^{-2\delta}$ and integrate the resulting equation on $B_1$. 
 
 \end{proof}

\begin{lem}[$L^2$ estimate] \label{L2}Let $\Omega$ be a bounded open set in $\R^n$. Let $b_i\in L^\infty(\Omega)$ with $$\|b_i\|_{L^\infty(\Omega)}\leq \Lambda,\quad  i=0,1,\dots,n. $$ Let $u\in L^2(\Omega)$ be a weak solution solution to $$\D u+\sum_{i=1}^n b_i\frac{\partial u}{\partial x_i}+b_0 u=f\quad\text{in }\Omega,$$ for some $f\in L^2(\Omega)$. Then for every $\tilde \Omega\Subset\Omega$ there exists $C=C(\tilde\Omega,\Lambda)$ such that $$\|u\|_{W^{2,2}(\tilde\Omega)}\leq C(\|f\|_{L^2(\Omega)}+\|u\|_{L^2(\Omega)}).$$  \end{lem}

\begin{lem}[Schauder estimate]\label{Schauder} Let $\Omega$ be a  bounded open set in $\R^n$. Let $b_i\in C^{0,\alpha}(\Omega)$ with $$\|b_i\|_{C^{0,\alpha}(\Omega)}\leq \Lambda,\quad  i=0,1,\dots,n. $$ Let $u$ be a classical solution to $$\D u+\sum_{i=1}^n b_i\frac{\partial u}{\partial x_i}+b_0 u=f\quad\text{in }\Omega,$$ for some $f\in C^{0,\alpha}(\Omega)$.  Then for every $\tilde \Omega\Subset\Omega$ there exists $C=C(\tilde\Omega,\Lambda)$ such that $$\|u\|_{C^{2,\alpha}(\tilde\Omega)}\leq C(\|f\|_{C^{0,\alpha}(\Omega)}+\|u\|_{C^0(\Omega)}).$$ Additionally,  if  $\Omega$ is regular,   $\partial \Omega$ has two components $\Gamma_1$ and $\Gamma_2$,  and if $u=0$ on $\Gamma_1$ then for $\tilde \Omega\Subset(\Omega\cup\Gamma_1)$ there exists $C=C(\tilde\Omega,\Lambda)$ such that $$\|u\|_{C^{2,\alpha}(\tilde\Omega)}\leq C(\|f\|_{C^{0,\alpha}(\Omega)}+\|u\|_{C^0(\Omega)}).$$   \end{lem}

\bibliographystyle{alpha} 
\bibliography{mybibfile}

 \end{document}